\documentclass[12pt]{article}

\usepackage{amsmath}
\usepackage{amsthm}
\usepackage{amssymb}
\usepackage[dvips]{color,graphicx}
\usepackage{psfrag}

\def\dim{{\rm dim\ }}

\def\A{{\mathcal A}}

\def\supp{\mathrm{supp}}

\def\h{\mathfrak{h}}

\def\NN{{\mathbb N}}

\def\RR{{\mathbb R}}

\def\conv{\mathrm{conv}}
\def\diam{\mathrm{diam}}

\def\XX#1{\mathcal{X}#1}

\def\g{\mathfrak{g}}
\def\ad{\mathrm{ad}}
\def\Exp{\mathrm{Exp}}
\def\P{\mathcal{P}}
\def\diam{\mathrm{diam\,}}
\def\K{\widehat{K}}

\def\GE#1{G_e\left(#1\right)}

\theoremstyle{plain}
\newtheorem{thm}{Theorem}[section]
\newtheorem{prop}[thm]{Proposition}
\newtheorem{dfn}[thm]{Definition}

\theoremstyle{definition}
\newtheorem*{rem}{Remark}

\begin{document}

\title{A support theorem for the X-ray transform on manifolds with plane covers}

\author{Norbert Peyerimhoff\footnote{Research supported by the University of Cyprus.} \and Evangelia Samiou$^*$}



\maketitle

\begin{abstract}
  This article is concerned with support theorems of the X-ray
  transform for manifolds with conjugate points. In
  particular, we prove a support theorem for $2$-step nilpotent Lie
  groups. An important ingredient of the proof is the concept of plane
  covers. We also provide examples of non-homogeneous $3$-dimensional
  simply connected manifolds with conjugate points which have support
  theorems.
\end{abstract}

\section{Introduction}\label{sec:1}

The classical X-ray transform associates to a continuous, sufficiently rapidly decreasing function on the Euclidean plane the values of its integrals along straight lines. J. Radon \cite{Ra1917} proved in 1917 an inversion formula for the X-ray transform and mentioned that, with a small modification, his inversion formula also holds for the hyperbolic plane. The X-ray transform and its variants are fundamental to computed tomography (see, e.g., \cite{Epstein2008, Feeman2015, Kuchment2014}).

The X-ray transform has also been studied in other geometries like Riemannian Symmetric and Damek-Ricci Spaces (see, e.g., \cite{Hel1965,Hel1981,Hel2011,Rou2006,Rou2010}) or so-called \emph{simple manifolds}, i.e., bounded manifolds with strictly convex boundary and without conjugate points (see, e.g., \cite{PSh1988,Shara1994,Kr2009,PSU2013}). The methods introduced for simple manifolds played also an important role in the boundary rigidity problem, i.e., the question whether the distance function between boundary points of a Riemannian manifold determines the Riemannian metric in the interior. See, e.g., \cite{PSU2014}, for a survey with a list of open problems.

In this article, we are particularly interested in \emph{support theorems} for the X-ray transform on spaces \emph{with conjugate points}. The classical support theorem in the Euclidean plane states the following for sufficiently decaying continuous functions $f \in C(\RR^2)$: if the integrals of $f$ over all geodesics avoiding a given closed ball $B$ vanish then we have $\supp f \subset B$ (see \cite[Theorem I.2.6]{Hel2011}). Moreover, this fact can be generalized to arbitrary compact sets $K \subset \RR^2$ (instead of balls). In this case the conclusion is $\supp f \subset \conv(K)$, where $\conv(K)$ is the convex hull of $K$ (see \cite[Cor. I.2.8]{Hel2011}). In particular the support theorem implies injectivity of the X-ray transform. It is natural to ask for analogous results in more general geometries. A support theorem for simple manifolds was proved in \cite{Kr2009}. We will use this for the Heisenberg group, even though the Heisenberg group has conjugate points.

All Riemannian manifolds $(M,g)$ in this paper are assumed to be simply connected, complete and non-compact. Moreover, we only consider functions with compact support to guarantee that the integrals over all \emph{escaping} geodesics are finite. Let us recall the definition of escaping geodesics (see, e.g., \cite{Woj1982}):

\begin{dfn} A unit speed geodesic $\gamma:\RR\to M$ is called \emph{escaping} if $\gamma$ is a proper map, i.e. for every compact set $K \subset M$ there exists $t_0 > 0$ such that $\gamma(t) \not\in K$ whenever $|t| \ge t_0$. We denote the set of all escaping geodesics of $M$ by $\GE{M}$. For any subset $M' \subset M$, let $\GE{M'}$ denote the set of all escaping geodesics of $M$ contained in $M'$. 
\end{dfn}

In plain words, an escaping geodesic $\gamma$ eventually leaves every given compact set. Let $C_c(M)$ be the space of all continuous functions with compact support. For every function $f\in C_c(M)$ the integral of $f$ over an escaping geodesic exists. In simply connected manifolds without conjugate points all geodesics are escaping. By the Cartan-Hadamard Theorem Riemannian manifolds of non-positive curvature do not have conjugate points. On the other hand, the Cartesian product $M = \RR \times S^2$ of the real line and the unit sphere has conjugate points and a geodesic $\gamma$ is escaping if and only if $\gamma'$ has a non-zero component in the direction of $\RR$.

\begin{dfn} The X-ray transform $\XX{f}$ of $f\in C_c(M)$ is the function on $\GE{M}$ given by
$$\XX{f}(\gamma):=\int_{-\infty}^\infty f(\gamma(t))dt\ .$$
\end{dfn}

Note that in the simply connected example $M = \RR \times S^2$ the
X-ray transform is not injective: Let $f \in C_c(M)$ be the constant
extension of a non-zero odd function $f_0 \in C_c(\RR)$. Then we have
$\XX{f}(\gamma) = 0$ for all escaping geodesics $\gamma$, but $f$ is
not zero.

\begin{dfn}\label{def:supthm}
We say that $(M,g)$ has a \emph{support theorem} if for every compact set $K \subset M$ there exists a compact set $\K \subset M$ with the following properties:
\begin{itemize} 
\item[(a)] For all functions $f \in C_c(M)$, $\XX{f}\vert_{\GE{M\backslash K}} = 0$
implies that $\supp f \subset \K$.
\item[(b)] If $K_n \subset M$ is a sequence of compact subsets of $M$ with $K_{n+1}\subset K_n$ and $\diam(K_n) \to 0$, then also \hbox{$\diam(\widehat{K_n}) \to 0$}.
\end{itemize}
\end{dfn}

We will be particularly interested in the concrete construction of a set $\K \subset M$ from a given compact set $K \subset M$ satisfying properties (a) and (b) in Definition \ref{def:supthm}. While in the Euclidean case we can choose $\K=\conv(K)$ to be the convex hull of $K$, it is not obvious how to choose $\K$ in more general geometries, in particular if there are conjugate points.

Let us discuss the properties (a) and (b) in more detail.

If $K$ is the support of a non-negative function $f\in C_c(M)$, then $\XX{f}(\gamma)>0$ for every geodesic $\gamma$ through the interior of $K$. It follows that $\K$ must contain the interior of $K$ for all compact sets $K$.

Property (b) guarantees that if the set $K$ becomes small then so does $\K$. In particular, if $K=\emptyset$, then $\diam\K=0$, hence $\K$ is empty or a singleton and any function $f$ with $\supp f\subset\K$ is zero. Thus (b) implies that the X-ray transform is injective.

\medskip

We aim to extend the strong support theorem in the Euclidean and hyperbollic plane to manifolds exhausted by such planes. 

\begin{dfn} A $2$-dimensional \emph{totally geodesic closed} submanifold $\Sigma$ of $(M,g)$ is called a \emph{plane in $M$} if $\Sigma$ is diffeomorphic to $\RR^2$ and if the induced metric on $\Sigma$ has constant non-positive curvature. A \emph{plane cover} of $M$ is a collection $\P$ of planes in $M$ such that
$$ M = \bigcup_{\Sigma \in \mathcal{P}} \Sigma. $$
\end{dfn}

Since a plane is closed all geodesics of $M$ contained in a plane are escaping. Hence a manifold with a plane cover has escaping geodesics through every point. 

Important classes of Riemannian manifolds $(M,g)$ have plane covers, most prominently symmetric spaces of noncompact type and higher rank. In Damek-Ricci spaces, and therefore in particular in non compact rank-$1$-symmetric spaces, every geodesic is contained in a hyperbolic plane, see \cite{Rou2010}. Hence these also have a plane cover. More generally homogeneous spaces containing a flat or a hyperbolic plane have a plane cover. In Proposition \ref{prop:homspace} we prove that such homogeneous spaces also have a support theorem. 

Proposition \ref{lm:nilp} implies that all $2$-step nilpotent Lie groups with left invariant metric of dimension $\ge 4$ have Euclidean plane covers. Nilpotent Lie groups of lower dimensions are abelian or, up to rescaling, isometrically isomorphic to the Heisenberg group $H^3$ (which coincides with the filiform group $L_3$, see \cite{KP2010}). By different techniques we show that $H^3$ also has a support theorem.

\begin{thm}
All simply connected $2$-step nilpotent Lie groups with left invariant metrics have a support theorem. In particular, the X-ray transform is injective.
\end{thm}

Products of Riemannian manifolds with only escaping geodesics have plane covers. However plane covers yielding support theorems do not always arise from homogeneity or a product structure. In section \ref{example3dim} we provide examples of a warped product metric on $\RR^3$ which is non-homogeneous and has conjugate points. Nonetheless, the metric has a plane cover yielding a support theorem. 
\bigskip

A more quantitative viewpoint on support theorems has been taken in \cite{PS2016}. The results there depend on the more stringent condition of {\em uniformly} esacaping geodesics, i.e. {\em every} geodesic $\gamma$ must have left the ball of radius $r$ around $\gamma(0)$ after a finite time $P(r)$ independent of the geodesic. If a manifold has a support theorem in the sense of Definition \ref{def:supthm} and $\sigma\colon\RR^+_0\to\RR^+_0$ is a function so that for all $p\in M$ and all $r\in\RR^+_0$ we have $\widehat{B_{\sigma(r)}(p)}\subset B_r(p)$, then the $\sigma$-support theorem as defined in \cite{PS2016} holds in the manifold.

\section{Plane covers}
Planes have a simple support theorem. By \cite{Hel1965}, for instance, the support of a compactly supported function on a plane lies in the convex hull of a set if the X-ray transform of the function vanishes on all geodesics avoiding the set.  
\begin{prop} \label{prop:planecover}
	Let $\P$ be a plane cover of $M$ and $K\subset M$ (not necessarily compact). Let $f \in C_c(M)$. If \hbox{$\XX{f}|_{\GE{M\setminus K} }=0$}, then 
	\begin{eqnarray}\label{defKhat} 
	\supp f \subset \K &:=& \{ x \in M \mid \forall \Sigma \in \P, x \in \Sigma\colon x \in \conv_\Sigma(K)\} \\
&=& \bigcap_{\Sigma\in\P}\left((M\setminus\Sigma)\cup\conv_\Sigma(K)\right), \nonumber
	\end{eqnarray}
	where for $A\subset M$, $\conv_{\Sigma}(A)$ denotes the convex hull of $A\cap\Sigma$ in $\Sigma$. In particular, $\XX{}$ is injective.
\end{prop}
\begin{proof}
	Let $f \in C_c(M)$ and $K \subset M$ be compact such that  $\XX{f}|_{\GE{M\setminus K}}=0$. Let $x \in M \backslash \K$. By the definition of $\K$ there is $\Sigma\in\P$ with $x\in\Sigma\setminus\conv_\Sigma(K\cap\Sigma)$. Let $f_0 = f\Big\vert_{\Sigma}$ and $K_0 = \Sigma \cap K$. Then $f_0 \in C_c(\Sigma)$ and $K_0 \subset \Sigma$ is compact. Moreover, $x \in \Sigma$ and $x \not\in {\conv}_\Sigma(K_0)$. After rescaling the metric by a constant factor $\Sigma$ is isometric to the Euclidean or hyperbolic plane. All geodesics $\gamma$ in $\Sigma$ are escaping and also geodesics in $M$. Therefore
$$\XX_\Sigma(f_0)|_{\Sigma\setminus K_0}=0\ .$$
Now we employ the Support Theorem for planes and conclude that
$$ f_0 \Big\vert_{\Sigma \backslash {\rm conv}_{\Sigma}(K_0)} \equiv 0, $$
and, in particular, $f(x) = f_0(x) = 0$.

By \eqref{defKhat}, if $K=\emptyset$ then $\K=\emptyset$, hence $\XX{}$ is injective.
\end{proof}

This does not automatically yield a support theorem because the set $\K$ defined by \eqref{defKhat} need not be precompact (i.e. bounded). In homogeneous spaces however, existence of a plane implies a support theorem.

\begin{prop}\label{prop:homspace} Let $M$ be a homogeneous space, $G$ 	a closed transitive subgroup of the isometry group of $M$, and $\P$ be a plane cover of $M$ containing all translates of at least one plane $\Sigma$. Then $\K$, defined in \eqref{defKhat}, is bounded for every compact subset $K \subset M$ and $M$ has a support theorem. 
\par
It follows that a homogeneous space which contains a plane has a support theorem.
\end{prop}

\begin{proof}
Let $\Sigma\subset M$ be a plane so that $g\Sigma\in\P$ for all $g\in G$. We may also assume that $p\in\Sigma\cap K\neq\emptyset$. Since the projection $G\to M$ is proper, the subset $G^K=\{g\in G\mid gp\in K\}\subset G$ is compact. The set
$$S_K:=\{(g,k_1,k_2)\mid g\in G^K, k_1,k_2\in K, g^{-1}k_1,g^{-1}k_2\in\Sigma\}\subset G^K\times K\times K$$
is compact and the function
$$\delta_K\colon S_K\to\RR^+_0\text{ with }\delta_K(g,k_1,k_2)=d^\Sigma(g^{-1}k_1,g^{-1}k_2)$$
is continuous. Therefore
$$D_K=\max_{g\in G^K}\diam^{g\Sigma} K\cap g\Sigma=\max\{d^\Sigma(g^{-1}k_1,g^{-1}k_2)\mid(g,k_1,k_2)\in S_K\}=\max\delta_K $$
exists and is finite. The convex hull of a subset of a plane has the same diameter as the subset itself. Hence
$$\diam^{g\Sigma}\conv_{g\Sigma}(K)=\diam^{g\Sigma} K\cap g\Sigma\leq D_K\text{ for all }g\in G^K\ .$$
By definition
$$\widehat{K}\subset\bigcup_{g\in G^K}\conv_{g\Sigma}(K)$$
and each slice $\conv_{g\Sigma}(K)$, $g\in G^K$, intersects $K$. It now follows that
\begin{equation}\label{diamKestimate}\diam\widehat{K}\leq 2D_K+\diam K\ .\end{equation}
\par
Finally, let $K_n$ be a sequence of compact sets as in Definition \ref{def:supthm} (b). If $K_n$ is empty for some $n$, then $K_m$ is empty for all $m\geq n$ and we have $D_{K_m}=\diam K_m=0$ hence $\diam\widehat{K_m}=0$ for all $m\geq n$, hence $\lim_{n\to\infty}\diam\widehat{K_n}=0$. If $K_n\neq\emptyset$ for all $n$, then for each $n\in\NN$ we have $g_n\in G^{K_n}$, $k_{1,n},k_{2,n}\in K_n$ so that
$$D_{K_n}=d^\Sigma(g_n^{-1}k_{1,n},g_n^{-1}k_{2,n})\ .$$
By compactness, we can assume that $\lim_{n\to\infty}g_n=g_\infty$ exists. Since\\ $\lim_{n\to\infty}\diam K_n =0$, the intersection $\bigcap_n K_n=\{q\}$ is a singleton and $\lim_{n\to\infty}k_{1,n}=q=\lim_{n\to\infty}k_{2,n}$. In particular, by continuity of $d^\Sigma$,
$$\lim_{n\to\infty}D_{K_n}=\lim_{n\to\infty}d^\Sigma(g_n^{-1}k_{1,n},g_n^{-1}k_{2,n})=d^\Sigma(g_\infty^{-1}q,g_\infty^{-1}q)=0\ .$$
By \eqref{diamKestimate},
$$\lim_{n\to\infty}\diam\widehat{K_n}\leq\lim_{n\to\infty}\left(2D_{K_n}+\diam K_n\right)=0\ .$$
\end{proof}

We now show that many nilpotent Lie groups with left invariant metric have Euclidean plane covers.

\begin{prop}\label{lm:nilp} Let $G$ be a simply connected $k$-step nilpotent Lie group with a left invariant metric. Let $\g$ be its Lie algebra and
$$0=\g_0\lhd\g_1\lhd\g_2\lhd\cdots\lhd\g_k=\g $$
be the upper central series of $\g$. Assume that for some $i\in\{1,\ldots,k\}$ we have
\begin{equation}\label{dimcondition}\dim\g_i/\g_{i-1}>1+\dim\g_{i-1}\ .\end{equation}
Then $G$ has a Euclidean plane cover and has a support theorem.
	
It follows immediatly that $G$ has a support theorem if $\dim\g_i\geq 2^i$ for some $i\in\{1,\ldots,k\}$, in particular if $\dim G\geq 2^k$, or if the center of $G$ has at least dimension $2$.
\end{prop}

\begin{proof} Let $i$ be as in the assumption of the theorem and let $\h_i=\g_{i-1}^\perp\cap\g_i$ be the orthogonal complement of $\g_{i-1}$ in $\g_i$. Then
$$\dim\h_i=\dim\g_i/\g_{i-1}>1+\dim\g_{i-1}\ .$$
For any $x\in\h_i$, the kernel of the linear map $\ad_x\colon\h_i\to\g_{i-1}$ is therefore at least $2$-dimensional. Hence there are $x,y\in\h_i\setminus\{0\}$ with $[x,y]=0$. By the Koszul formula for left invariant vector fields we have for any $z\in\g$, 
\begin{equation}\label{koszulformula}
2\langle\nabla_{x}y,z\rangle=\langle\underbrace{[x,y]}_{=0},z\rangle -\langle x,\underbrace{[y,z]}_{\in\g_{i-1}}\rangle -\langle y,\underbrace{[x,z]}_{\in\g_{i-1}}
\rangle = 0\ .\end{equation}
This shows that $\nabla_{x}x=\nabla_xy=\nabla_yy=0$, hence $F=\{\Exp(rx+sy)\mid r,s\in\RR\}$ is a totally geodesic flat submanfold of $G$ and therefore the translates of $F$ form a Euclidean plane cover of $G$. Since the exponential map of a nilpotent Lie group is a diffeomorphism, $F$ is a closed submanifold.
\end{proof}

\begin{rem}A nilpotent group can have a plane cover even if the dimension condition \eqref{dimcondition} is violated. A $k$-step nilpotent group must have at least dimension $n=k+1$, the groups with this minimal dimension are called filiform, \cite{KP2010}. For example (in the notation of \cite{KP2010}, p 1592), let $L_n$ be the metric Lie algebra with orthonormal basis $\{X_1,X_2,\ldots,X_n\}$, and Lie bracket so that $[X_1,X_i]=X_{i+1}$ and all other brackets trivial. Then if $n>3$, \eqref{koszulformula} shows that any pair $X_i,X_j$ with $1\leq i\leq j\leq n$, $i+1<j$, spans a plane in $L_n$.
\end{rem}

\section{A support theorem for the Heisenberg group}

We now show that the Heisenberg group $H^3$ has a support theorem. In particular, the X-ray transform on $H^3$ is injective.

\begin{thm} For every compact subset $K\subset H^3$, there is a compact subset $\K\subset H^3$ such that if $f\in C_c(H^3)$ is so that $\int_{\gamma}f=0$ for all geodesics $\gamma$ in $H^3\setminus K$, then $f|_{H^3\setminus\K}=0$.
\end{thm}

\begin{proof}
We will use the Lie group exponential
$$\Exp\colon\RR^3\cong\h^3\cong H^3$$
as coordinates for $H^3$, with the last coordinate representing the center and the first two its orthogonal complement. The Campbell-Hausdorff formula for a $2$-step nilpotent Lie group reads
$$\Exp(x)\Exp(a)=\Exp\left(x+a+\frac{1}{2}[x,a]\right)$$
for $x,a$ in the Lie algebra. In the case of the Heisenberg group this becomes
$$\Exp(x,y,z)\Exp(a,b,c)=\Exp\left(x+a,y+b,z+c+\frac{xb-ya}{2}\right)\ .$$
In the sequel we will suppress the exponential map in the notation. For the geodesics in the Heisenberg group see \cite{BTV1995}, p31. We will only need horizontal geodesics, in particular those of the form
$$\gamma_x(t)=\left(x,t,\frac{xt}{2}\right)=(x,0,0)(0,t,0)$$
and those obtained from $\gamma_x$ by rotations about the $z$-axis with an angle $\alpha$, which we will denote by $\gamma_{x,\alpha}$. These geodesics lie outside the paraboloid
$$P=\left\{(u,v,w)\mid w\geq\frac{u^2+v^2}{4}\right\} $$
with exception of one point. The geodesic $\gamma_x$ touches the boundary of $P$ in the point $\gamma_x(x)=\left (x,x,\frac{x^2}{2}\right )$. Clearly every point on the boundary of $P$ is the unique intersection point of $\partial P$ with a suitable geodesic $\gamma_{x,\alpha}$.    

\begin{center}
 \psfrag{a}{$a$}
 \psfrag{x}{$x$}
 \psfrag{x0}{$x_0$}
 \psfrag{y}{$y$}
 \psfrag{z}{$z$}
 \psfrag{P}{\Large$P$}
 \psfrag{Ph}{\Large$P^h$}
 \psfrag{B2d(q)}{\Large$B_{2\delta}(q)$}
 \psfrag{gx0}{$\gamma_{x_0}$}
 \psfrag{gx0z}{$\gamma_{x_0}^{z_0}$}
  \includegraphics[width=0.6\textwidth]{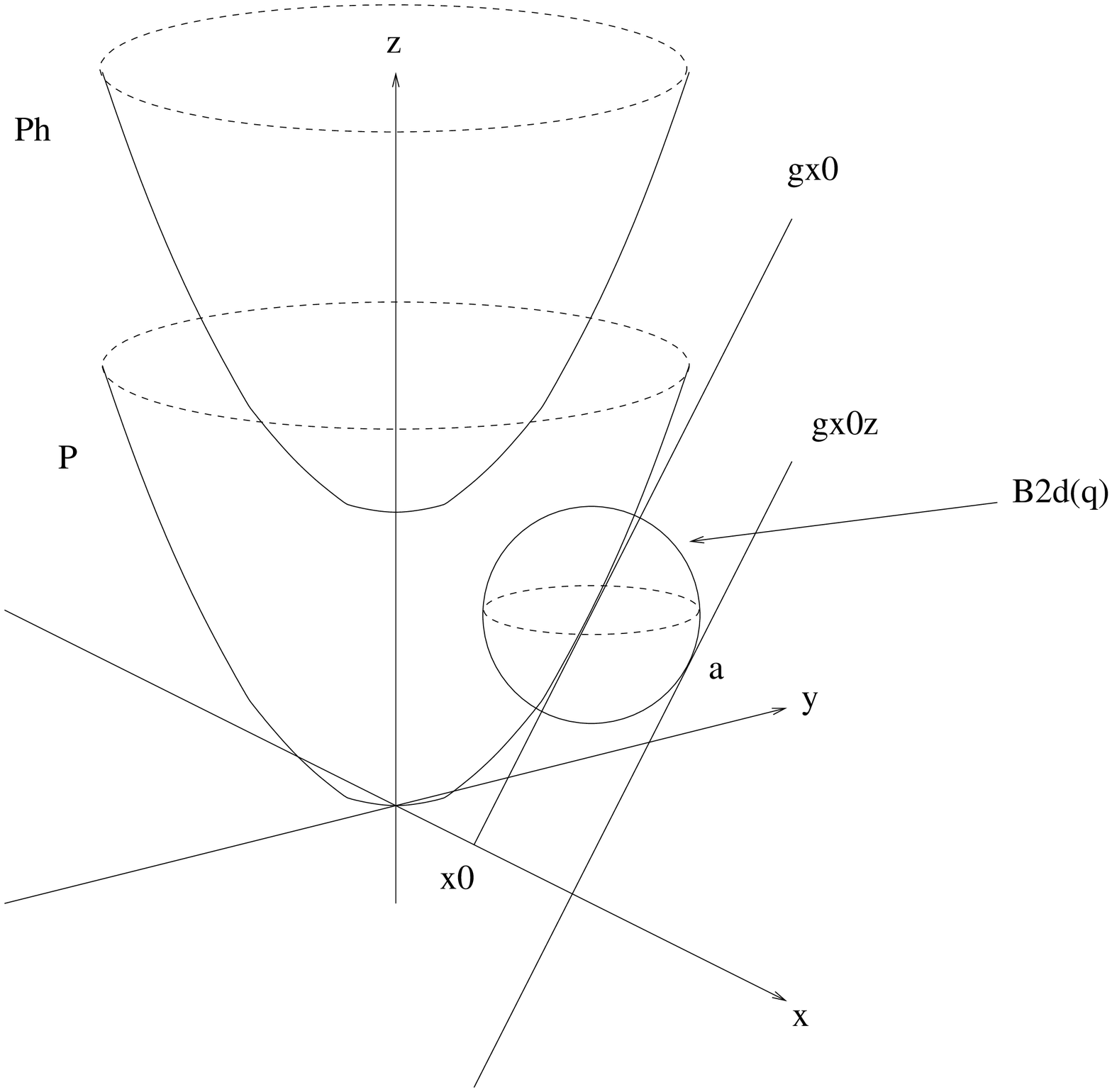}
\end{center}

Let $f$ be a continuous function on $H^3$, with compact support. Let $L\subset H^3$ be a compact subset so that the integral of $f$ over any geodesic avoiding $L$ vanishes. We can enclose $L$ by two shifted paraboloids,
\begin{multline*}
L\subset (P-(0,0,h_-))\cap (-P+(0,0,h_+))= \\ = \left\{(u,v,w)\mid \frac{u^2+v^2}{4}-h_-\leq w\leq h_+-\frac{u^2+v^2}{4}\right\}
\end{multline*}
for suitable $h_-,h_+\in\RR$. We will show that the support of $f$ then also lies in the intersection of these paraboloids. By symmetry it clearly suffices to do this for one paraboloid only.  Shifting $f$ in direction of the center, i.e. the $z$-axis, if necessary, we may assume that $L$ lies in some shifted paraboloid and that the support of $f$ lies in the paraboloid $P$,
$$L\subset P^h:=P+(0,0,h)\subset P\supset\supp f\text{ , for some }h>0\ ,$$
and that the support of $f$ intersects the boundary of $P$. Rotating about the $z$-axis if necessary, we can also assume that such an intersection point $q$ lies on a geodesic $\gamma_{x_0}$,
$$\supp f\cap\partial P\ni q=\gamma_{x_0}(x_0)$$
for some $x_0\in\RR$. Let $\delta\in\RR$ be so that $B_{2\delta}(q)\cap P^h=\emptyset$ and $\overline{B_{\rho}(q)}$ is simple for all $\rho\leq 2\delta$, i.e. has strictly convex boundary and the Riemannian exponential map is a diffeomorphism where defined. For $z\in(-\infty,0]$ consider the geodesics
$$\gamma_{x_0}^z(t)=\left(x_0,t,\frac{x_0t}{2}+z\right)\ .$$
Among these only $\gamma_{x_0}=\gamma_{x_0}^0$ intersects $P$. There is $z_0<0$ so that $\gamma_{x_0}^{z_0}$ intersects $\partial B_{2\delta}(q)$ in one point $a$.

The geodesics in $\Gamma:=\{\gamma_{x_0}^z\mid z\in[z_0,0]\}$ do not intersect $P^h$. By Theorem 3.1 from \cite{PS2016}, geodesics in 2-step nilpotent Lie groups are uniformly escaping. Thus there is $T\in\RR^+$, such that for all geodesics $\gamma$ with  $\gamma(0)\in B_{2\delta}(q)$ and all $t>T$, we have $\gamma(t)\not\in L\cup(\supp f\setminus B_{\delta}(q))$. On the other hand, since $[-T,T]$ is compact, there is $\epsilon>0$ so that no geodesic in the $\epsilon$-neighbourhood $\A$ of $\Gamma$ intersects $L\cup(\supp f\setminus B_{\delta}(q))$. $\Gamma\subset\A$ is a deformation retract.

In particular any geodesic in $\A$ can be deformed to the point $a$. From Theorem 1, \cite{Kr2009}, we infer that the function $f$ vanishes on the union of the intersections of the geodesics in $\A$ with $B_{\delta}(q)$, which is a neighbourhood of $q$, hence $q$ can not lie in the closure of $f^{-1}(\RR\setminus\{0\})$.  
\end{proof} 

\section{Non-homogeneous examples in dimension $3$ with conjugate points}\label{example3dim}

The previous examples are all homogeneous. We finish this paper presenting $3$-dimen\-sional non-homogeneous examples with conjugate points admitting Euclidean and hyperbolic plane covers so that for every compact set $K$ we have $\K$ bounded.

Let $g$ be a Riemannian metric on $\RR^3$ given in cylindrical coordinates
$(t,r,\alpha)\in\RR\times [0,\infty) \times [0,2\pi)$ by
\begin{equation}\label{eq:metric3dim} 
g = dt^2 + dr^2 + f(r,t)^2d\alpha^2, 
\end{equation}
where  $f\colon\RR\times[0,\infty)\to\RR$ is a smooth function satisfying
$$f(r,t)=\left\{\begin{array}{l@{\text{ if }}l}
\sin(r)&0\leq r\leq\frac{3\pi}{4}\\
\phi(r) & \frac{3\pi}{4}\leq r\leq\pi\\
2+r-\pi & \pi\leq r\leq 2\pi\\
2+r-\pi & 2\pi\leq r\text{ and }t\leq 0\\
>2+r-\pi & 2\pi\leq r\text{ and }t>0
\end{array}\right.$$
where $\phi\colon\left[\frac{3\pi}{4},\pi\right]\to(0,2]$ is so that $f$ becomes smooth.

The submanifold $F=\{(0,r,\alpha)\mid r\in[0,\infty),\alpha \in [0,2\pi)\}$ is locally isometric to the round unit sphere for $r<\frac{3}{4}\pi$. Inside the cylinder $r\leq 2\pi$, the manifold is isometric to a product, inside the cylinder $r\leq\frac{\pi}{2}$ this is the product of the round hemisphere with $\RR$.

For $\alpha\in[0,\pi)$ let $E_\alpha:=\{(t,r,\alpha),(t,r,\alpha+\pi),\mid t\in\RR, r\in[0,\infty)\}$. Then $\mathcal{P} = \{ E_\alpha \mid \alpha \in [0,\pi)\}$ is a plane cover. The geodesics emanating from the origin are the same as those of the standard metric of $\RR^3$. In particular $g$ is complete. For $t\in\RR$, $S_t=\{(t,r,\alpha)\mid0\leq r\leq\frac{\pi}{2},0\leq\alpha<2\pi\}$ is a totally geodesic submanifold with boundary, isometric to a standard hemisphere. Antipodal points $(t,\frac{\pi}{2},\alpha)$ and $(t,\frac{\pi}{2},\alpha+\pi)$ on the boundary of $S_t$ are conjugate. To see this, observe that the embeddings \begin{equation}
\label{y34kk}
\{(t,r,\alpha)\mid r\leq\pi\}\cong\RR\times S^2_+\subset\RR\times S^2\end{equation} are isometric. The rotations of $S^2$ about the axis corresponding under \eqref{y34kk} to a pair of antipodal points in the boundary of $S_+$ provide a one parameter family of minimizing geodesics joining two such points. The metric is not homogeneous, in fact $g$ is flat in the region $\{r>\pi,t<0\}$ and isometric to $S^2\times\RR$ in the region $\{r\leq\frac{\pi}{2}\}$. In the region $\{t<0\text{ or }r<2\pi\}$, $g$ is isometric to $F\times\RR$, but clearly, this isometry does not hold globally.

Let $K\subset\RR^3$ be compact and let $D=\max\{\|k\|\mid k\in K\}=\max\{\sqrt{t^2+r^2}\mid (t,r,\alpha)\in K\}$, note that the distance of a point from the origin with respect to $g$ is the same as the euclidean distance. Thus $K\subset\overline{B_D(0)}$. Clearly, for all $E_\alpha\in\P$ we have $K\cap E_\alpha\subset\overline{B_D(0)}$, hence
$$\K\subset\bigcup_{\alpha}\conv_{E_\alpha}(K)\subset\overline{B_D(0)}\ .$$

In order to obtain a similar example with a hyperbolic plane cover, we replace the metric \eqref{eq:metric3dim} with
\begin{equation}\label{eq:metric3dimhyp} 
g = dt^2 + e^{2t}\left(dr^2 + f(r,t)^2d\alpha^2\right)\ . 
\end{equation}
The submanifolds $E_\alpha$ as above are now totally geodesic hyperbolic planes (with curvature $-1$) and $\P=\{ E_\alpha \mid \alpha \in [0,\pi)\}$ is a hyperbolic plane cover. The hemispheres $S_t$ from above are not totally geodesic. As before the points $(t,\frac{\pi}{2},\alpha)$ and $(t,\frac{\pi}{2},\alpha+\pi)$ are conjugate. The maps in \eqref{y34kk} are still isometric provided $\RR\times S^2$ carries the metric $dt^2+e^{2t}g_{S^2}$ where $g_{S^2}$ is the standard metric on $S^2$. As above rotations of $S^2$ provide a one parameter family of minimizing geodesics joining such two points.

\noindent
Norbert Peyerimhoff\\
Department of Mathematical Sciences, Durham University, Science Laboratories\\
South Road, Durham, DH1 3LE, UK

\bigskip

\noindent
Evangelia Samiou\\
Department of Mathematics and Statistics, University of Cyprus\\
P.O. Box 20537, 1678 Nicosia, Cyprus

\end{document}